\documentclass[a4paper,twoside,11pt,english,intlimits]{article}

\usepackage{amsmath,amsthm,amsfonts,amssymb}
\usepackage[english]{babel}
\usepackage[utf8]{inputenc}
\usepackage[T1]{fontenc}
\usepackage{times,euler,euscript,stmaryrd}
\usepackage{fancyhdr,titlesec,url,enumerate}
\usepackage[all]{xy}
\usepackage{hyperref}
\usepackage{ifthen}
\usepackage{tikz}
\usepackage{comment}
\usepackage{wasysym}

\CompileMatrices
\usetikzlibrary{calc}

\DeclareMathAlphabet{\mathscr}{T1}{pzc}{m}{it} 

\titleformat{\section}[block]{\scshape\filcenter\Large}{\thesection.}{.5em}{}
\titleformat{\subsection}[block]{\bfseries\filcenter\large}{\thesubsection.}{.5em}{\medskip}
\titleformat{\subsubsection}[runin]{\bfseries}{\thesubsubsection.}{.5em}{}[.]

\titlespacing{\subsubsection}{0pt}{10pt}{.5em}

\newtheoremstyle{ntheorem}%
	{\topsep}{\topsep}{\itshape}{0pt}{\bfseries}{.}{.5em}%
	{\thmnumber{#2.\hspace{.5em}}\thmname{#1}\thmnote{ (#3)}}
	
\newtheoremstyle{ndefinition}%
	{\topsep}{\topsep}{\normalfont}{0pt}{\bfseries}{.}{.5em}%
	{\thmnumber{#2.\hspace{.5em}}\thmname{#1}\thmnote{ (#3)}}

\theoremstyle{ntheorem}
  	
  	\newtheorem{proposition}[subsubsection]{Proposition}

\theoremstyle{ndefinition}

\fancypagestyle{plain}{
\fancyhf{}\fancyfoot[LC,RC]{}
\fancyfoot[LE,RO]{$\thepage$}

}

\UseTips
\SelectTips{eu}{11}

\newdir{ >}{{}*!/-10pt/@{>}}
\newdir{ -}{{}*!/-10pt/@{}}
\newdir{> }{{}*!/+10pt/@{>}}

\makeatletter

\xyletcsnamecsname@{dir4{}}{dir{}}
\xydefcsname@{dir4{-}}{\line@ \quadruple@\xydashh@}
\xydefcsname@{dir4{.}}{\point@ \quadruple@\xydashh@}
\xydefcsname@{dir4{~}}{\squiggle@ \quadruple@\xybsqlh@}
\xydefcsname@{dir4{>}}{\Tttip@}
\xydefcsname@{dir4{<}}{\reverseDirection@\Tttip@}

\xydef@\quadruple@#1{%
	\edef\Drop@@{%
		\dimen@=#1\relax
		\dimen@=.5\dimen@
		\A@=-\sinDirection\dimen@
		\B@=\cosDirection\dimen@
		\setboxz@h{%
			\setbox2=\hbox{\kern3\A@\raise3\B@\copy\z@}%
			\dp2=\z@ \ht2=\z@ \wd2=\z@ \box2
			\setbox2=\hbox{\kern\A@\raise\B@\copy\z@}%
			\dp2=\z@ \ht2=\z@ \wd2=\z@ \box2
			\setbox2=\hbox{\kern-\A@\raise-\B@\copy\z@}%
			\dp2=\z@ \ht2=\z@ \wd2=\z@ \box2
			\setbox2=\hbox{\kern-3\A@\raise-3\B@ \noexpand\boxz@}%
			\dp2=\z@ \ht2=\z@ \wd2=\z@ \box2
		}%
		\ht\z@=\z@ \dp\z@=\z@ \wd\z@=\z@ \noexpand\styledboxz@
	}%
}

\xydef@\Tttip@{\kern2pt \vrule height2pt depth2pt width\z@
	\Tttip@@ \kern2pt \egroup
	\U@c=0pt \D@c=0pt \L@c=0pt \R@c=0pt \Edge@c={\circleEdge}%
	\def\Leftness@{.5}\def\Upness@{.5}%
	\def\Drop@@{\styledboxz@}\def\Connect@@{\straight@{\dottedSpread@\jot}}}
	
\xydef@\Tttip@@{%
	\dimen@=.25\dimen@
 	\B@=\cosDirection\dimen@
	\setboxz@h\bgroup\reverseDirection@\line@ \wdz@=\z@ \ht\z@=\z@ \dp\z@=\z@
	{\vDirection@(1,-1)\xydashl@ \xyatipfont\char\DirectionChar}%
	{\vDirection@(1,+1)\xydashl@ \xybtipfont\char\DirectionChar}%
}

\xydef@\ar@form{
	\ifx \space@\next \expandafter\DN@\space{\xyFN@\ar@form}%
	\else\ifx ^\next \DN@ ^{\xyFN@\ar@style}\edef\arvariant@@{\string^}%
	\else\ifx _\next \DN@ _{\xyFN@\ar@style}\edef\arvariant@@{\string_}%
	\else\ifx 0\next \DN@ 0{\xyFN@\ar@style}\def\arvariant@@{0}%
	\else\ifx 1\next \DN@ 1{\xyFN@\ar@style}\def\arvariant@@{1}%
	\else\ifx 2\next \DN@ 2{\xyFN@\ar@style}\def\arvariant@@{2}%
	\else\ifx 3\next \DN@ 3{\xyFN@\ar@style}\def\arvariant@@{3}%
	\else\ifx 4\next \DN@ 4{\xyFN@\ar@style}\def\arvariant@@{4}%
	\else\ifx \bgroup\next \let\next@=\ar@style
	\else\ifx [\next \DN@[##1]{\ar@modifiers{[##1]}}
	\else\ifx *\next \DN@ *{\ar@modifiers}%
	\else\addLT@\ifx\next \let\next@=\ar@slide
	\else\ifx /\next \let\next@=\ar@curveslash
	\else\ifx (\next \let\next@=\ar@curveinout 
	\else\addRQ@\ifx\next \addRQ@\DN@{\ar@curve@}%
	\else\addLQ@\ifx\next \addLQ@\DN@{\xyFN@\ar@curve}%
	\else\addDASH@\ifx\next \addDASH@\DN@{\defarstem@-\xyFN@\ar@}%
	\else\addEQ@\ifx\next \addEQ@\DN@{\def\arvariant@@{2}\defarstem@-\xyFN@\ar@}%
	\else\addDOT@\ifx\next \addDOT@\DN@{\defarstem@.\xyFN@\ar@}%
	\else\ifx :\next \DN@:{\def\arvariant@@{2}\defarstem@.\xyFN@\ar@}%
	\else\ifx ~\next \DN@~{\defarstem@~\xyFN@\ar@}%
	\else\ifx !\next \DN@!{\dasharstem@\xyFN@\ar@}%
	\else\ifx ?\next \DN@?{\ar@upsidedown\xyFN@\ar@}%
	\else \let\next@=\ar@error
	\fi\fi\fi\fi\fi\fi\fi\fi\fi\fi\fi\fi\fi\fi\fi\fi\fi\fi\fi\fi\fi\fi\fi \next@}

\makeatother

\newcommand{\fl}{\to}
\newcommand{\dfl}{\Rightarrow}

\newcommand{\qfl}{\xymatrix@1@C=10pt{\ar@4 [r] &}}

\renewcommand{\phi}{\varphi}
\renewcommand{\epsilon}{\varepsilon}

\definecolor{orange}{rgb}{1,0.55,0}
\definecolor{vert}{rgb}{0,0.45,0}

\newcommand{\ifthen}[2]{\ifthenelse{#1}{#2}{}}

\renewcommand{\fl}{\rightarrow}

\newcommand{\ofl}[1]{\overset{#1}{\fl}}

\newcommand{\odfl}[1]{\overset{#1}{\dfl}}

\newcommand{\kar}[1]{\mathrm{Kar}(#1)}

\begin{document}

\quad

\vspace{-2cm}

\begin{center}
\begin{Large}
\textsc{Linear polygraphs applied to categorification}
\end{Large}

\vskip+5pt

\textbf{Cl\'ement Alleaume}

Univ Lyon, Université Claude Bernard Lyon 1, CNRS UMR 5208,

Institut Camille Jordan, 43 blvd. du 11 novembre 1918, F-69622 Villeurbanne cedex, France \textsf{clement.alleaume@univ-st-etienne.fr}
\end{center}

\begin{small}\begin{minipage}{14cm}
\noindent\textbf{Abstract --} 
We introduce two applications of polygraphs to categorification problems. We compute first, from a coherent presentation of an $n$-category, a coherent presentation of its Karoubi envelope. For this, we extend the construction of Karoubi envelope to $n$-polygraphs and linear $(n,n-1)$-polygraphs. The second problem treated in this paper is the construction of Grothendieck decategorifications for $(n,n-1)$-polygraphs. This construction yields a rewriting system presenting for example algebras categorified by a linear monoidal category. We finally link quasi-convergence of such rewriting systems to the uniqueness of direct sum decompositions for linear $(n-1,n-1)$-categories.

\end{minipage}\end{small}


\section{Introduction}

Karoubi envelopes of categories were introduced as a way to classify the idempotents of a category \cite{Bunge79}. The Karoubi envelope of a category $\mathcal{C}$ is an explicitly defined completion $\kar{\mathcal{C}}$ of $\mathcal{C}$ splitting all idempotents. In particular, if all idempotents of the category $\mathcal{C}$ are split, the category $\kar{\mathcal{C}}$ is equivalent to $\mathcal{C}$.  In this work, we focus on the presentations, expressed using the structure of polygraph, of Karoubi envelopes of monoidal categories by generators and relations. From a finite polygraph that presents a monoidal category we construct a finite polygraph presenting its Karoubi envelope. We wish to generalize the notion of Karoubi envelope to higher-dimensional (strict) categories.

Karoubi envelopes are used to construct categorifications of algebras. Categorification is a process giving from an algebra $A$ a linear monoidal category whose Grothendieck group is isomorphic to $A$ \cite{Cr,Maz}. An example of a categorification is the Khovanov homology \cite{Kh}, a categorification of the Jones polynomials. This categorification was used to give a new proof of Milnor's conjecture \cite{Ras}. Algebras like Hecke algebras \cite{EW} can be categorified by the Karoubi envelope of a diagrammatic category, that is a monoidal category in which the morphism spaces are depicted by string diagrams, the 0-composition by horizontal concatenation and the 1-composition by vertical concatenation. Khovanov conjectured that the Karoubi envelope of some diagrammatic category \cite{Kho} categorifies the Heisenberg algebra. More generally, we are interested in categories defined by generators and relations and the categorifications induced by the Karoubi envelopes of such categories.

Polygraphs were independently introduced by Street and Burroni \cite{Street87,Burroni93} as systems of generators and oriented relations, or rewriting rules, for higher-dimensional categories. For $n \geqslant 1$, an $(n+1)$-polygraph is a presentation of an $n$-category by generators and relations. In particular, a 3-polygraph with only one 0-cell is a presentation of a monoidal category. A linear variation of polygraphs was introduced in \cite{All} to present linear categories. A linear $(n,p)$-category is an $n$-category with a linear structure on its set of $k$-cells for any $k\geqslant p$. A linear $(n+1,n)$-polygraph is a rewriting system on the $n$-cells of a linear $(n,n)$-category. This rewriting system presents a  linear $(n,n)$-category. In particular, a $(3,2)$-linear polygraph is a presentation of a linear monoidal category.

A coherent presentation of an $n$-category $\mathcal{C}$ is a data made of an $(n+1)$-polygraph $\Sigma$ presenting $\mathcal{C}$ and a family of $(n+2)$-cells $\Sigma_{n+2}$ such that the quotient of the free $(n+1,n)$-category over $\Sigma$ by the congruence generated by $\Sigma_{n+2}$ is aspherical. Coherence problems appear for instance in the construction of resolutions called polygraphic resolutions \cite{GaussentGuiraudMalbos15}.

In this article, we define a generalization of the Karoubi envelope for $n$-categories and construct the Karoubi envelope of an $(n+1)$-polygraph. We also give an adaptation of this definition to linear $(n+1,n)$-polygraphs. Our first goal is to generalize the notion of a Grothendieck decategorification to linear $(n,n)$-categories. Our second goal is to construct from a linear $(n,n)$-category, presented by a linear $(n+1,n)$-polygraph a linear $(n,n-1)$-polygraph, presenting its Grothendieck decategorification. With this work, we can present first the Karoubi envelope of a linear $(n,n)$-category from a linear $(n+1,n)$-polygraph presenting this linear $(n,n)$-category. The next step is to present the Grothendieck decategorification of this Karoubi envelope to solve categorification problems.

In the first section of this paper, we recall the notions of Karoubi envelope and polygraph. Then, in the second section, we define a notion of a Karoubi envelope for polygraphs and give a coherence result on this construction \ref{result}. We then recall the definition of linear $(n,p)$-polygraphs and give similar results on their Karoubi envelopes \ref{reslin}. Finally, in the last section, we define the Grothendieck decategorification of an $(n,n)$-category, already defined by Mazorchuk \cite{Maz2} for $n=2$, and the Grothendieck decategorification of a linear $(n+1,n)$-polygraph. We prove that the Grothendieck decategorification of a linear $(n+1,n)$-polygraph $\Sigma$ presents the Grothendieck decategorification of the linear $(n,n)$-category presented by $\Sigma$, see \ref{presentation}. We conclude with a result which allows to decide if decompositions as a direct sum of indecomposable $(n-1)$-cells in a linear $(n,n)$-category are unique up to isomorphism or not \ref{thKS}. Answering negatively this question implies the linear $(n,n)$-category is not Krull-Schmidt.

\section{Karoubi envelopes and polygraphs}

We recall in this section the notions of Karoubi envelope and polygraph.

\subsection{Karoubi envelope of an $n$-category}

\subsubsection{Categorical notations} In an $n$-category, for any $0 \leqslant k<n$, we denote the $k$-composition by $\star_k$. For all $0 \leqslant i<j \leqslant n-1$ the following equality, called \emph{exchange relation}, holds:
\begin{equation}\label{exch} (u \star_i v) \star_j (u' \star_i v')=(u \star_j u') \star_i (v \star_j v'). \end{equation}
An \emph{$(n,p)$-category} is an $n$-category whose $k$-cells are invertible for the $(k-1)$-composition for any $p<k\leqslant n$. We denote by $\mathbf{Cat}_n$ the category of $n$-categories and~$n$-functors. This category has a terminal object $I_n$ with only one $k$-cell for~$0 \leqslant k \leqslant n$.  An \emph{$(n,p)$-category} is an $n$-category whose $k$-cells are invertible for the $k$-composition for any $p<k\leqslant n$. We denote by $\mathbf{Cat}_{n,p}$ the category of $(n,p)$-categories.

\subsubsection{Idempotents and Karoubi envelope} Let $n\geqslant 1$ be an integer and $\mathcal{C}$ be an $n$-category. An \emph{idempotent} of $\mathcal{C}$ is an $n$-cell $e$ of $\mathcal{C}$ such that $e \star_{n-1} e=e$. Note that the $(n-1)$-source and the $(n-1)$-target of an idempotent are necessarily equal. If there are no integer $k<n-1$ and idempotents $e'$ and $e''$ such that, $e=e'\star_ke''$, we say that the idempotent $e$ is \emph{minimal}. We say that the idempotent $e$ is \emph{split} if there exists an $(n-1)$-cell $A$ of $\mathcal{C}$, an $n$-cell $p$ from $s_{n-1}(e)$ to $A$ and an $n$-cell $p$ from $A$ to $s_{n-1}(e)$ such that:
\begin{itemize}
\item $p \star_{n-1} i=e$,
\item $i \star_{n-1} p=1_{s_{n-1}(e)}$.
\end{itemize}

The \emph{Karoubi envelope} of the $n$-category $\mathcal{C}$ is the $n$-category $\kar{\mathcal{C}}$ such that:
\begin{itemize}
\item $\kar{\mathcal{C}}$ has the same $k$-cells than $\mathcal{C}$ for $k<n-1$,
\item $\kar{\mathcal{C}}$ has an $(n-1)$-cell $A_e$ from $s_{n-2}(e)$ to $t_{n-2}(e)$ for each idempotent $e$ of $\mathcal{C}$,
\item for $k<n-1$, for each $k$-composable idempotents $e$ and $e'$ of $\mathcal{C}$, we have $A_e \star_k A_{e'}=A_{e \star_k e'}$,
\item $\kar{\mathcal{C}}$ has an $n$-cell $\alpha (e,f,e')$ from $A_e$ to $A_{e'}$ for each triple $(e,f,e')$ of $n$-cells of $\mathcal{C}$ such that $e$ and $e'$ are idempotents verifying $f=e \star_{n-1} f \star_{n-1} e'$,
\item for $k<n-1$, for each pairs of $k$-composable idempotents $(e_1,e_2)$ and $(e'_1,e'_2)$ of $\mathcal{C}$ and each $k$-composable $n$-cells $f_1$ and $f_2$ of $\mathcal{C}$ such that $\alpha (e_1,f_1,e'_1)$ and $\alpha (e_2,f_2,e'_2)$ are defined, we have $\alpha (e_1,f_1,e'_1) \star_k \alpha (e_2,f_2,e'_2)=\alpha (e_1\star_k e_2,f_1\star_k f_2,e'_1\star_k e'_2)$,
\item for each $(n-1)$-composable $n$-cells $f$ and $g$ of $\mathcal{C}$ and each triple $(e,e',e'')$ of idempotents of $\mathcal{C}$ such that $f=e \star_{n-1} f \star_{n-1} e'$ and $g=e' \star_{n-1} g \star_{n-1} e''$, we have $\alpha (e,f,e') \star_{n-1} \alpha (e',g,e'')=\alpha (e,f\star_{n-1} g,e'')$.
\end{itemize}

The $n$-category $\mathcal{C}$ is isomorphic to a sub $n$-category of $\kar{\mathcal{C}}$. An injective $n$-functor $F$ from $\mathcal{C}$ to $\kar{\mathcal{C}}$ is defined by:
\begin{itemize}
\item $F_k$ is an identity for any $k<n-1$,
\item $F_{n-1}(x)=A_{1_x}$ for any $(n-1)$-cell $x$ of $\mathcal{C}$,
\item $F_n(a)=\alpha (1_{s_{n-1}(a)},a,1_{s_{n-1}(a)})$ for any $n$-cell $a$ of $\mathcal{C}$.
\end{itemize}

From now on, we will consider any $k$-cell of $\mathcal{C}$ as a $k$-cell of $\kar{\mathcal{C}}$ by applying this injective $n$-functor.

\subsubsection{Remark} In the $n$-category $\kar{\mathcal{C}}$, the identity $n$-cell of $A_e$ is the $n$-cell $\alpha (e,e,e)$.

\subsubsection{The canonical surjection $n$-functor}\label{surj} Let $\mathcal{C}$ be an $n$-category. There is a surjective $n$-functor $\mathrm{CS}$ from $\kar{\mathcal{C}}$ to $\mathcal{C}$ defined by:
\begin{itemize}
\item the restriction of $\mathrm{CS}$ to $\mathcal{C}$ is an identity,
\item $\mathrm{CS}_{n-1}(A_e)=s_{n-1}(e)$ for any idempotent $e$ of $\mathcal{C}$,
\item $\mathrm{CS}_n(\alpha (e,f,e'))=f$ for any minimal idempotent $e$ of $\mathcal{C}$ for each triple $(e,f,e')$ of $n$-cells of $\mathcal{C}$ such that $e$ and $e'$ are idempotents verifying $f=e \star_{n-1} f \star_{n-1} e'$.
\end{itemize}
We call this $n$-functor the \emph{canonical surjection $n$-functor} from $\kar{\mathcal{C}}$ to $\mathcal{C}$.

\subsubsection{Proposition} Let $\mathcal{C}$ be an $n$-category and let $\kar{\mathcal{C}}$ be its Karoubi envelope. In the $n$-category $\kar{\mathcal{C}}$, all idempotents are split.

\begin{proof} All idempotents of $\kar{\mathcal{C}}$ can be written $\alpha (e',e,e')$ where $e$ and $e'$ are idempotents of $\mathcal{C}$ verifying $e' \star_{n-1} e\star_{n-1} e'=e$. This implies:
$$e' \star_{n-1} e= e' \star_{n-1} e\star_{n-1} e=e' \star_{n-1} e\star_{n-1} e' \star_{n-1} e\star_{n-1} e'=e\star_{n-1} e\star_{n-1} e'=e\star_{n-1} e'.$$
Thus, we obtain:
$$e' \star_{n-1} e= e' \star_{n-1} e'\star_{n-1} e=e' \star_{n-1} e\star_{n-1} e'=e,$$
$$e \star_{n-1} e'= e \star_{n-1} e'\star_{n-1} e'=e' \star_{n-1} e\star_{n-1} e'=e.$$
Let us now prove that the idempotent $\alpha (e',e,e')$ of $\kar{\mathcal{C}}$ is split. Because of the equalities $e' \star_{n-1} e\star_{n-1} e=e=e \star_{n-1} e\star_{n-1} e'$, the $n$-cells $\alpha (e',e,e)$ and $\alpha (e,e,e')$ of $\kar{\mathcal{C}}$ are well-defined. The equalities $\alpha (e',e,e) \star_{n-1} \alpha (e,e,e')=\alpha (e',e,e')$ and $\alpha (e,e,e') \star_{n-1} \alpha (e',e,e)=\alpha (e,e,e)=1_{A_e}$ conclude the proof. \end{proof}

\subsection{Polygraphs}

In this section, we recall the definition of $n$-graphs. We also recall the constructions of globular extensions and $(n,p)$-polygraphs given in \cite{Metayer08}.

\subsubsection{Definition of $n$-graphs} An \emph{$n$-graph} in a category~$\mathbf{C}$ is a diagram in $\mathbf{C}$:
\begin{gather*}
\begin{array}{c}
\tikz[scale=0.9]{
\node at (0,0) {$G_0$};
\draw[color=black, ->] (1.5,0.25) -- (0.5,0.25);
\node at (1,0.5) {$s_0$};
\draw[color=black, ->] (1.5,-0.25) -- (0.5,-0.25);
\node at (1,-0.5) {$t_0$};
\node at (2,0) {$G_1$};
\draw[color=black, ->] (3.5,0.25) -- (2.5,0.25);
\node at (3,0.5) {$s_1$};
\draw[color=black, ->] (3.5,-0.25) -- (2.5,-0.25);
\node at (3,-0.5) {$t_1$};
\node at (4,0) {$\cdots$};
\draw[color=black, ->] (5.5,0.25) -- (4.5,0.25);
\node at (5,0.5) {$s_{n-2}$};
\draw[color=black, ->] (5.5,-0.25) -- (4.5,-0.25);
\node at (5,-0.5) {$t_{n-2}$};
\node at (6,0) {$G_{n-1}$};
\draw[color=black, ->] (7.5,0.25) -- (6.5,0.25);
\node at (7,0.5) {$s_{n-1}$};
\draw[color=black, ->] (7.5,-0.25) -- (6.5,-0.25);
\node at (7,-0.5) {$t_{n-1}$};
\node at (8,0) {$G_n$};
} \end{array}
\end{gather*}
such that for any $1 \leqslant k \leqslant n-1$, we have~$s_{k-1} \circ s_k=s_{k-1} \circ t_k$ and~$t_{k-1} \circ s_k=t_{k-1} \circ t_k$. Those relations are called the \emph{globular relations}. We just call an $n$-graph in $\mathbf{Set}$ an $n$-graph.

The elements of $G_k$ are called~\emph{$k$-cells}. The maps $s_k$ and~$t_k$ are respectively called~\emph{$k$-source} and~\emph{$k$-target maps}. For any $l$-cell~$u$ of $G$ with~$l > k+1$, we respectively denote by $s_k(u)$ and~$t_k(u)$ the $k$-cells $(s_k \circ \cdots  \circ s_{l-1})(u)$ and~$(t_k \circ \cdots  \circ t_{l-1})(u)$.

A \emph{morphism of $n$-graphs} $F$ from $G$ to~$G'$ is a collection~$(F_k: G_k \rightarrow G'_k)$ of maps such that, for every~$0 < k \leqslant n$, the following diagrams commute:
\begin{gather*}
\begin{array}{c}
\tikz[scale=0.9]{
\node at (0,0) {$G_{k-1}$};
\draw[color=black, ->] (1.5,0) -- (0.5,0);
\node at (1,0.5) {$s_{k-1}$};
\node at (2,0) {$G_k$};
\draw[color=black, ->] (0,-0.3) -- (0,-1.7);
\node at (-0.5,-1) {$F_{k-1}$};
\draw[color=black, ->] (2,-0.3) -- (2,-1.7);
\node at (2.5,-1) {$F_k$};
\node at (0,-2) {$G'_{k-1}$};
\draw[color=black, ->] (1.5,-2) -- (0.5,-2);
\node at (1,-2.25) {$s'_{k-1}$};
\node at (2,-2) {$G'_k$};
} \end{array}
\qquad
\begin{array}{c}
\tikz[scale=0.9]{
\node at (0,0) {$G_{k-1}$};
\draw[color=black, ->] (1.5,0) -- (0.5,0);
\node at (1,0.5) {$t_{k-1}$};
\node at (2,0) {$G_k$};
\draw[color=black, ->] (0,-0.3) -- (0,-1.7);
\node at (-0.5,-1) {$F_{k-1}$};
\draw[color=black, ->] (2,-0.3) -- (2,-1.7);
\node at (2.5,-1) {$F_k$};
\node at (0,-2) {$G'_{k-1}$};
\draw[color=black, ->] (1.5,-2) -- (0.5,-2);
\node at (1,-2.25) {$t'_{k-1}$};
\node at (2,-2) {$G'_k$};
} \end{array}
\end{gather*}

\subsubsection{Globular extensions} The category~$\mathbf{Cat}_n^+$ of $n$-categories with a globular extension is defined by the following pullback diagram:
\begin{gather*}
\begin{array}{c}
\tikz[scale=0.9]{
\node at (0,0) {$\mathbf{Cat}_n^+$};
\draw[color=black, ->] (0.5,0) -- (2,0);
\node at (2.9,0) {$\mathbf{Grph}_{n+1}$};
\draw[color=black, ->] (2.7,-0.5) -- (2.7,-2);
\node at (2.7,-2.5) {$\mathbf{Grph}_n$};
\draw[color=black, ->] (0,-0.5) -- (0,-2);
\node at (0,-2.5) {$\mathbf{Cat}_n$};
\draw[color=black, ->] (0.5,-2.5) -- (2,-2.5);
\node at (2.7,-2.5) {$\mathbf{Grph}_n$};
\draw[color=black] (0.3,-0.5) -- (0.55,-0.5) -- (0.55,-0.25);
\node at (1.25,-3) {$\mathcal{U}_n$};
\node at (3.2,-1.25) {$\mathcal{U}_n^G$};
} \end{array}
\end{gather*}
where $\mathcal{U}_n^G$ is the functor from $\mathbf{Grph}_{n+1}$ to~$\mathbf{Grph}_n$ associating to each $(n+1)$-graph its underlying~$n$-graph by eliminating the $(n+1)$-cells. The objects of $\mathbf{Cat}_n^+$ are of the form $(\mathcal{C},\Gamma)$ where $\mathcal{C}$ is an $n$-category and $\Gamma$ a set of $(n+1)$-cells.

\subsubsection{Free constructions over a globular extension} Let $(\mathcal{C},\Gamma)$ be an object of $\mathbf{Cat}_n^+$. The \emph{free~$(n+1)$-category} over~$(\mathcal{C},\Gamma)$ is the $(n+1)$-category whose underlying~$n$-category is $\mathcal{C}$ and whose~$(n+1)$-cells are the compositions of elements of $\Gamma$ and elements of the form $1_u$ where $u$ is in $\mathcal{C}_n$. The free functor from $\mathbf{Cat}_n^+$ to~$\mathbf{Cat}_{n+1}$ is denoted by $\mathcal{F}_{n+1}^W$. The \emph{free~$(n+1,n)$-category} over~$(\mathcal{C},\Gamma)$ is the $(n+1,n)$-category obtained by adding to the free~$(n+1)$-category over~$(\mathcal{C},\Gamma)$ formal inverses to its $(n+1)$-cells for the $n$-composition.

\subsubsection{Homotopy bases} A globular extension $\Gamma$ of the $n$-category~$\mathcal{C}$ is called a \emph{homotopy basis} of~$\mathcal{C}$ if for any $n$-sphere $(f,g)$ of $\mathcal{C}$, the free $(n+1,n)$-category over~$(\mathcal{C},\Gamma)$ has an $(n+1)$-cell from $f$ to $g$.

\subsubsection{Polygraphs} The category $\mathbf{Pol}_0$ of 0-polygraphs is the category of sets and the functor~$\mathcal{F}_0$ from $\mathbf{Pol}_0$ to~$\mathbf{Cat}_0$ is the identity functor. Let us assume the category~$\mathbf{Pol}_n$ of $n$-polygraphs and the functor~$\mathcal{F}_n$ from $\mathbf{Pol}_n$ to~$\mathbf{Cat}_n$ are defined. The category~$\mathbf{Pol}_{n+1}$ is defined by the following pullback diagram:
\begin{gather*}
\begin{array}{c}
\tikz[scale=0.9]{
\node at (0,0) {$\mathbf{Pol}_{n+1}$};
\draw[color=black, ->] (0.7,0) -- (5,0);
\node at (5.9,0) {$\mathbf{Grph}_{n+1}$};
\draw[color=black, ->] (5.7,-0.5) -- (5.7,-2);
\node at (5.7,-2.5) {$\mathbf{Grph}_n$};
\draw[color=black, ->] (0,-0.5) -- (0,-2);
\node at (0,-2.5) {$\mathbf{Pol}_n$};
\draw[color=black, ->] (0.5,-2.5) -- (2,-2.5);
\node at (2.5,-2.5) {$\mathbf{Cat}_n$};
\draw[color=black, ->] (3,-2.5) -- (5,-2.5);
\draw[color=black] (0.3,-0.5) -- (0.55,-0.5) -- (0.55,-0.25);
\node at (4,-3) {$\mathcal{U}_n$};
\node at (1,-3) {$\mathcal{F}_n$};
\node at (6.2,-1.25) {$\mathcal{U}_n^G$};
\node at (-0.5,-1.25) {$\mathcal{U}_n^P$};
\node at (2.85,0.5) {$\mathcal{U}_{n+1}^{GP}$};
} \end{array}
\end{gather*}
We denote by $\mathcal{F}_{n+1}^P$ the unique functor making the following diagram commutative:
\begin{gather*}
\begin{array}{c}
\tikz[scale=0.9]{
\node at (-3,2) {$\mathbf{Pol}_{n+1}$};
\draw[color=black, ->] (-2.3,1.5) -- (-0.6,0.4);
\draw[color=black, ->] (-3,1.5) -- (-3,-2);
\node at (-3,-2.5) {$\mathbf{Pol}_n$};
\draw[color=black, ->] (-2.3,-2.5) -- (-0.5,-2.5);
\node at (-1.4,-3) {$\mathcal{F}_n$};
\draw[color=black, ->] (-2.3,2) -- (2,0.4);
\node at (-1.45,0.45) {$\mathcal{F}_{n+1}^P$};
\node at (0,0) {$\mathbf{Cat}_n^+$};
\draw[color=black, ->] (0.5,0) -- (2,0);
\node at (2.9,0) {$\mathbf{Grph}_{n+1}$};
\draw[color=black, ->] (2.7,-0.5) -- (2.7,-2);
\node at (2.7,-2.5) {$\mathbf{Grph}_n$};
\draw[color=black, ->] (0,-0.5) -- (0,-2);
\node at (0,-2.5) {$\mathbf{Cat}_n$};
\draw[color=black, ->] (0.5,-2.5) -- (2,-2.5);
\node at (2.7,-2.5) {$\mathbf{Grph}_n$};
\draw[color=black] (0.3,-0.5) -- (0.55,-0.5) -- (0.55,-0.25);
\node at (1.25,-3) {$\mathcal{U}_n$};
\node at (3.2,-1.25) {$\mathcal{U}_n^G$};
\node at (-3.5,-0.25) {$\mathcal{U}_n^P$};
\node at (0.5,1.45) {$\mathcal{U}_{n+1}^{GP}$};
} \end{array}
\end{gather*}
The functor~$\mathcal{F}_{n+1}$ is defined as the following composite:
\begin{gather*}
\begin{array}{c}
\tikz[scale=0.9]{
\node at (0,0) {$\mathbf{Pol}_{n+1}$};
\draw[color=black, ->] (0.7,0) -- (2.7,0);
\node at (3.2,0) {$\mathbf{Cat}_n^+$};
\draw[color=black, ->] (3.7,0) -- (5.7,0);
\node at (6.4,0) {$\mathbf{Cat}_{n+1}$};
\node at (1.7,0.5) {$\mathcal{F}_{n+1}^P$};
\node at (4.7,0.5) {$\mathcal{F}_{n+1}^W$};
} \end{array}
\end{gather*}
Given an $n$-polygraph $\Sigma$, we call $\Sigma^*$ the free $n$-category over $\Sigma$.

Similarly, we construct the category $\mathbf{Pol}_{n,p}$ and the functor$\mathcal{F}_{n,p}$ by induction on $n \geqslant p$. We define first the category~$\mathbf{Cat}_{n,p}^+$ of $(n,p)$-categories with a globular extension by the following pullback diagram:
\begin{gather*}
\begin{array}{c}
\tikz[scale=0.9]{
\node at (0,0) {$\mathbf{Cat}_{n,p}^+$};
\draw[color=black, ->] (0.5,0) -- (2,0);
\node at (2.9,0) {$\mathbf{Grph}_{n+1}$};
\draw[color=black, ->] (2.7,-0.5) -- (2.7,-2);
\node at (2.7,-2.5) {$\mathbf{Grph}_n$};
\draw[color=black, ->] (0,-0.5) -- (0,-2);
\node at (0,-2.5) {$\mathbf{Cat}_n$};
\draw[color=black, ->] (0.5,-2.5) -- (2,-2.5);
\node at (2.7,-2.5) {$\mathbf{Grph}_n$};
\draw[color=black] (0.3,-0.5) -- (0.55,-0.5) -- (0.55,-0.25);
\node at (1.25,-3) {$\mathcal{U}_n$};
\node at (3.2,-1.25) {$\mathcal{U}_{n,p}^G$};
} \end{array}
\end{gather*}
with $\mathcal{U}_{n,p}^G$ the forgetful functor from $\mathbf{Cat}_{n,p}^+$ to $\mathbf{Grph}_n$. Next, we define $\mathbf{Pol}_{n,n}=\mathbf{Pol}_n$ and $\mathcal{F}_{n,n}=\mathcal{F}_n$. Assuming the category $\mathbf{Pol}_{n,p}$ and the functor $\mathcal{F}_{n,p}$ from $\mathbf{Pol}_{n,p}$ to $\mathbf{Cat}_{n,p}$ are constructed, we now define $\mathbf{Pol}_{n+1,p}$ by the pullback diagram:
\begin{gather*}
\begin{array}{c}
\tikz[scale=0.9]{
\node at (0,0) {$\mathbf{Pol}_{n+1,p}$};
\draw[color=black, ->] (0.7,0) -- (5,0);
\node at (5.9,0) {$\mathbf{Grph}_{n+1}$};
\draw[color=black, ->] (5.7,-0.5) -- (5.7,-2);
\node at (5.7,-2.5) {$\mathbf{Grph}_n$};
\draw[color=black, ->] (0,-0.5) -- (0,-2);
\node at (0,-2.5) {$\mathbf{Pol}_{n,p}$};
\draw[color=black, ->] (0.5,-2.5) -- (1.9,-2.5);
\node at (2.5,-2.5) {$\mathbf{Cat}_{n,p}$};
\draw[color=black, ->] (3.1,-2.5) -- (5,-2.5);
\draw[color=black] (0.3,-0.5) -- (0.55,-0.5) -- (0.55,-0.25);
\node at (4,-3) {$\mathcal{U}_n$};
\node at (1,-3) {$\mathcal{F}_{n,p}$};
\node at (6.2,-1.25) {$\mathcal{U}_{n,p}^G$};
\node at (-0.5,-1.25) {$\mathcal{U}_n^P$};
\node at (2.85,0.5) {$\mathcal{U}_{n+1}^{GP}$};
} \end{array}
\end{gather*}
and define $\mathcal{F}_{n,p}^P$ as the unique functor making the following diagram commutative:
\begin{gather*}
\begin{array}{c}
\tikz[scale=0.9]{
\node at (-3,2) {$\mathbf{Pol}_{n+1,p}$};
\draw[color=black, ->] (-2.3,1.5) -- (-0.6,0.4);
\draw[color=black, ->] (-3,1.5) -- (-3,-2);
\node at (-3,-2.5) {$\mathbf{Pol}_{n,p}$};
\draw[color=black, ->] (-2.3,-2.5) -- (-0.5,-2.5);
\node at (-1.4,-3) {$\mathcal{F}_{n,p}$};
\draw[color=black, ->] (-2.3,2) -- (2,0.4);
\node at (-1.45,0.45) {$\mathcal{F}_{n+1,p}^P$};
\node at (0,0) {$\mathbf{Cat}_n^+$};
\draw[color=black, ->] (0.5,0) -- (2,0);
\node at (2.9,0) {$\mathbf{Grph}_{n+1}$};
\draw[color=black, ->] (2.7,-0.5) -- (2.7,-2);
\node at (2.7,-2.5) {$\mathbf{Grph}_n$};
\draw[color=black, ->] (0,-0.5) -- (0,-2);
\node at (0,-2.5) {$\mathbf{Cat}_n$};
\draw[color=black, ->] (0.5,-2.5) -- (2,-2.5);
\node at (2.7,-2.5) {$\mathbf{Grph}_n$};
\draw[color=black] (0.3,-0.5) -- (0.55,-0.5) -- (0.55,-0.25);
\node at (1.25,-3) {$\mathcal{U}_n$};
\node at (3.2,-1.25) {$\mathcal{U}_{n,p}^G$};
\node at (-3.5,-0.25) {$\mathcal{U}_{n,p}^P$};
\node at (0.5,1.45) {$\mathcal{U}_{n+1}^{GP}$};
} \end{array}
\end{gather*}
to finally define $\mathcal{F}_{n+1,p}$ as the following composite:
\begin{gather*}
\begin{array}{c}
\tikz[scale=0.9]{
\node at (0,0) {$\mathbf{Pol}_{n+1,p}$};
\draw[color=black, ->] (0.7,0) -- (2.6,0);
\node at (3.2,0) {$\mathbf{Cat}_{n,p}^+$};
\draw[color=black, ->] (3.7,0) -- (5.6,0);
\node at (6.4,0) {$\mathbf{Cat}_{n+1,p}$};
\node at (1.7,0.5) {$\mathcal{F}_{n+1,p}^P$};
\node at (4.7,0.5) {$\mathcal{F}_{n+1,p}^W$};
} \end{array}
\end{gather*}

Given an $(n,p)$-polygraph $\Sigma$, we call $\Sigma^\top$ the free $(n,p)$-category over $\Sigma$.

\subsubsection{Presentation of an $n$-category} Let $\mathcal{C}$ be an $n$-category. An $(n+1)$-polygraph $\Sigma$ is said to \emph{present} the $n$-category $\mathcal{C}$ if $\mathcal{C}$ is isomorphic to $\Sigma_n^*/\Sigma_{n+1}$. Two $(n+1)$-polygraphs are \emph{Tietze equivalent} if they present the same $n$-category. A \emph{coherent presentation} of the $n$-category $\mathcal{C}$ is an $(n+2,n)$-polygraph $\Sigma$ such that the $(n+1)$-polygraph $\Sigma_{n+1}$ is a presentation of $\mathcal{C}$ and the set $\Sigma_{n+2}$ is a homotopy basis of $\Sigma_{n+1}^{\top}$.

\section{Coherent presentation of a Karoubi envelope}

In this section, we define the Karoubi envelope of an $(n+1)$-polygraph and construct a coherent presentation of the Karoubi envelope of an $n$-category $\mathcal{C}$ from a coherent presentation of $\mathcal{C}$. We then recall the definition of linear an $(n,p)$-polygraph and the Karoubi envelope of a linear $(n+1,n)$-polygraph. We finally give a construction of a coherent presentation of the Karoubi envelope of a linear $(n+1,n)$-category $\mathcal{C}$ from a coherent presentation of $\mathcal{C}$.

\subsection{Presentation of Karoubi envelopes}

\subsubsection{Karoubi envelope of an $(n+1)$-polygraph}\label{karpol} Let $\Sigma$ be an $(n+1)$-polygraph. The Karoubi envelope of $\Sigma$ is the $(n+1)$-polygraph $\kar{\Sigma}$ defined by:
\begin{itemize}
\item $\kar{\Sigma}_k=\Sigma_k$ for $k<n-1$,
\item $\kar{\Sigma}_{n-1}=\Sigma_{n-1}\cup \{A_e|~\text{$e$ is a minimal idempotent of $\mathcal{C}$}\}$,
\item for each minimal idempotent $e$ of $\mathcal{C}$, we have $s_{n-2}(A_e)=s_{n-2}(e)$ and $t_{n-2}(A_e)=t_{n-2}(e)$,
\item $\kar{\Sigma}_n=\Sigma_n\cup \{p_e,i_e|~\text{$e$ is a minimal idempotent of $\mathcal{C}$}\}$,
\item for each minimal idempotent $e$ of $\mathcal{C}$, we have $s_{n-1}(p_e)=s_{n-1}(e)$ and $t_{n-1}(p_e)=A_e$,
\item for each minimal idempotent $e$ of $\mathcal{C}$, we have $s_{n-1}(i_e)=A_e$ and $t_{n-1}(i_e)=t_{n-1}(e)$,
\item $\kar{\Sigma}_{n+1}=\Sigma_{n+1}\cup \{\pi_e,\iota_e|~\text{$e$ is a minimal idempotent of $\mathcal{C}$}\}$,
\item for each minimal idempotent $e$ of $\mathcal{C}$, we have $s_n(\pi_e)=e$ and $t_n(\pi_e)=p_e \star_n i_e$,
\item for each minimal idempotent $e$ of $\mathcal{C}$, we have $s_n(\iota_e)=i_e \star_n p_e$ and $t_n(\iota_e)=1_{s_n(e)}$.
\end{itemize}

\begin{proposition}\label{pol} Let $\mathcal{C}$ be an $n$-category presented by an $(n+1)$-polygraph $\Sigma$. The Karoubi envelope of $\mathcal{C}$ is presented by the $(n+1)$-polygraph $\kar{\Sigma}$. \end{proposition}

\begin{proof}Let $\kar{\mathcal{C}}$ be the Karoubi envelope of the $n$-category $\mathcal{C}$. For $k<n-1$, the $n$-category $\kar{\mathcal{C}}$ has the same $k$-cells than $\mathcal{C}$ and $\kar{\Sigma}_k=\Sigma_k$. Then, the $(n+1)$-polygraph $\kar{\Sigma}$ presents an $n$-category with the same $k$-cells than $\kar{\mathcal{C}}$. Let us now prove that $\kar{\Sigma}$ presents an $n$-category with the same $(n-1)$-cells than $\kar{\mathcal{C}}$. Let $e$ be an idempotent of $\mathcal{C}$ and let us write:
$$e=e_0 \star_{k_1} e_1 \star_{k_2} \cdots \star_{k_m} e_m$$
where all $e_i$ are minimal idempotents and all $k_i$ are integer smaller than $n-1$. We can write:
$$A_e=A_{e_0} \star_{k_1} A_{e_1} \star_{k_2} \cdots \star_{k_m} A_{e_m}$$
which corresponds to an $(n-1)$-cell in the $n$-category presented by $\kar{\Sigma}$. What remains to prove is that the $n$-category $\kar{\Sigma}^*/\kar{\Sigma}_{n+1}$ has the same $n$-cells and relations on $n$-cells than $\kar{\mathcal{C}}$. There is an injective $n$-functor $F$ from $\kar{\Sigma}^*/\kar{\Sigma}_{n+1}$ to $\kar{\mathcal{C}}$ defined by:
\begin{itemize}
\item $F$ sends each $n$-cell of $\Sigma_k$ onto its representative in $\mathcal{C}$,
\item for each minimal idempotent $e$ of $\mathcal{C}$, the $n$-functor $F$ sends the $n$-cell $p_e$ onto $\alpha (1_{s_{n-1}(e)},e,e)$,
\item for each minimal idempotent $e$ of $\mathcal{C}$, the $n$-functor $F$ sends the $n$-cell $i_e$ onto $\alpha (e,e,1_{s_{n-1}(e)})$.
\end{itemize}
Let us prove that the $n$-functor $F$ is surjective. Let $e$ be an idempotent of $\mathcal{C}$ and let us write again the decomposition into minimal idempotents:
$$e=e_0 \star_{k_1} e_1 \star_{k_2} \cdots \star_{k_m} e_m.$$
We then have the decompositions:
$$\alpha (1_{s_{n-1}(e)},e,e)=\alpha (1_{s_{n-1}(e_0)},e_0,e_0) \star_{k_1} \alpha (1_{s_{n-1}(e_1)},e_1,e_1)  \star_{k_2} \cdots \star_{k_m} \alpha (1_{s_{n-1}(e_m)},e_m,e_m),$$
$$\alpha (e,e,1_{s_{n-1}(e)})=\alpha (e_0,e_0,1_{s_{n-1}(e_0)}) \star_{k_1} \alpha (e_1,e_1,1_{s_{n-1}(e_1)})  \star_{k_2} \cdots \star_{k_m} =\alpha (e_m,e_m,1_{s_{n-1}(e_m)}).$$
Thus, the $n$-cells $\alpha (1_{s_{n-1}(e)},e,e)$ and $\alpha (e,e,1_{s_{n-1}(e)})$ are images by $F$ of $n$-cells of $\kar{\Sigma}^*/\kar{\Sigma}_{n+1}$. Let now $\alpha (e,f,e')$ be an $n$-cell of the $\kar{\mathcal{C}}$ such that $e$ and $e'$ are idempotents of $\mathcal{C}$. We have:
$$\alpha (e,f,e')=\alpha (e,e,1_{s_{n-1}(e)})\star_{n-1} \alpha (1_{s_{n-1}(e)},f,1_{s_{n-1}(e)})\star_{n-1} \alpha (1_{s_{n-1}(e)},e',e').$$
Thus, the $n$-cell $\alpha (e,f,e')$ is the image by $F$ of an $n$-cell of $\kar{\Sigma}^*/\kar{\Sigma}_{n+1}$. This concludes the proof. \end{proof}

\subsubsection{Example}\label{part1} Let $M$ be the monoid presented by the following 2-polygraph $\Sigma$ defined by:
\begin{itemize}
\item $\Sigma_0$ has only one 0-cell,
\item $\Sigma_1$ has two 1-cells $a$ and $b$,
\item $\Sigma_2$ has a 2-cell $\alpha$ from $aba$ to $a$.
\end{itemize}
The monoid $M$ has two minimal idempotents: $ab$ and $ba$. Thus, by \ref{pol}, the Karoubi envelope of $M$ is presented by the 2-polygraph $\kar{\Sigma}$ defined by:
\begin{itemize}
\item $\kar{\Sigma}_0=\{ O,X,Y\}$,
\item $\kar{\Sigma}_1=\{ O \ofl{a} O,O \ofl{b} O,O \ofl{p_X} X,X \ofl{i_X} O,O \ofl{p_Y} Y,Y \ofl{i_Y} O\}$,
\item $\kar{\Sigma}_2=\{ aba \odfl{\alpha} a,p_Xi_X \odfl{\pi_X} ab,i_Xp_X \odfl{\iota_X} 1_X,p_Yi_Y \odfl{\pi_Y} ba,i_Yp_Y \odfl{\iota_Y} 1_Y\}$.
\end{itemize}

\subsubsection{Karoubi envelope of a globular extension} Let $\mathcal{C}$ be an $n$-category. Let $\Gamma$ be a globular extension of $\mathcal{C}$. For each $(n+1)$-cell $A$ of $\Gamma$ with $n$-source $f$ and $n$-target $g$, we define the set $\mathrm{CS}^{-1}(A)$ as a set containing an $(n+1)$-cell from $f'$ to $g'$ for each parallel $n$-cells $f'$ and $g'$ of $\kar{\mathcal{C}}$ such that $\mathrm{CS}(f')=f$ and $\mathrm{CS}(g')=g$ with $\mathrm{CS}$ being the canonical surjection $n$-functor from $\kar{\mathcal{C}}$ to $\mathcal{C}$. The Karoubi envelope of the globular extension $\Gamma$ is the globular extension of $\kar{\mathcal{C}}$ defined by:
$$\kar{\Gamma}=\bigcup_{A \in \Gamma}\mathrm{CS}^{-1}(A).$$

\subsubsection{Theorem}\label{result} Let $\mathcal{C}$ be an $n$-category and let $(\Sigma,\Sigma_{n+2})$ be a coherent presentation of $\mathcal{C}$. The \linebreak $(n+2,n)$-polygraph ($\kar{\Sigma},\kar{\Sigma_{n+2}})$ is a coherent presentation of the Karoubi envelope of $\mathcal{C}$.

\begin{proof} We proceed in four steps.

\noindent {\bf Step 1.} Let $f$ and $g$ be parallel $(n+1)$-cells of $\kar{\Sigma}^\top$ such that there is an $(n+2)$-cell $A$ from $\mathrm{CS}(f)$ to $\mathrm{CS}(g)$ in $\Sigma_{n+2}$. We prove that there is an $(n+2)$-cell from $g$ to $f$ in $\kar{\Sigma_{n+2}}^\top$. There is an $(n+2)$-cell of $\mathrm{CS}^{-1}(A)$ from $f$ to $g$. The inverse of this $(n+2)$-cell is in $\kar{\Sigma_{n+2}}^\top$.

\noindent {\bf Step 2.} Let $f$ and $g$ be parallel $(n+1)$-cells of $\kar{\Sigma}^\top$ such that there is an $(n+2)$-cell $A$ from $\mathrm{CS}(f)$ to $\mathrm{CS}(g)$ in $\Sigma_{n+2}$. Let $f'$ and $g'$ be parallel $(n+1)$-cells of $\kar{\Sigma}^\top$ such that there is an $(n+2)$-cell $A'$ from $\mathrm{CS}(f')$ to $\mathrm{CS}(g')$ in $\Sigma_{n+2}$. Let us assume the $(n+1)$-cells $f \star_k f'$ and $g \star_k g'$ for an integer $k<n$. We prove that there is an $(n+2)$-cell from $f \star_k f'$ to $g \star_k g'$ in $\kar{\Sigma_{n+2}}^\top$. There is an $(n+2)$-cell of $\mathrm{CS}^{-1}(A)$ from $f$ to $g$ and an $(n+2)$-cell of $\mathrm{CS}^{-1}(A')$ from $f'$ to $g'$. Their $k$-composition is in $\kar{\Sigma_{n+2}}^\top$.

\noindent {\bf Step 3.} Let $f$, $g$ and $h$ be parallel $(n+1)$-cells of $\kar{\Sigma}^\top$ such that there is an $(n+2)$-cell $A$ from $\mathrm{CS}(f)$ to $\mathrm{CS}(g)$ in $\Sigma_{n+2}$ and an $(n+2)$-cell $B$ from $\mathrm{CS}(g)$ to $\mathrm{CS}(h)$ in $\Sigma_{n+2}$. We prove that there is an $(n+2)$-cell from $f$ to $h$ in $\kar{\Sigma_{n+2}}^\top$. There is an $(n+2)$-cell of $\mathrm{CS}^{-1}(A)$ from $f$ to $g$ and an $(n+2)$-cell of $\mathrm{CS}^{-1}(B)$ from $g$ to $h$. Their $n$-composition is in $\kar{\Sigma_{n+2}}^\top$.

\noindent {\bf Step 4.} Let $f$ and $g$ be parallel $(n+1)$-cells of $\kar{\Sigma}^\top$. We prove that there is an $(n+2)$-cell from $f$ to $g$ in $\kar{\Sigma_{n+2}}^\top$. Because $\Sigma_{n+2}$ is a homotopy basis of $\Sigma^\top$, there is an $(n+2)$-cell from $\mathrm{CS}(f)$ to $\mathrm{CS}(g)$ obtained by compositions and inversions of $(n+2)$-cells of $\Sigma_{n+2}$ and identities $(n+2)$-cells.  By steps 1, 2 and 3, this allows us to construct an $(n+2)$-cell from $f$ to $g$ in $\kar{\Sigma_{n+2}}^\top$. \end{proof}

\subsubsection{Remark} The free $(n+1)$-category over the Karoubi envelope of an $(n+1)$-polygraph $\Sigma$ is not the Karoubi envelope of the free $(n+1)$-category $\Sigma^*$. Indeed, the only idempotents of $\Sigma^*$ are the identities $(n+1)$-cells. Thus, $\kar{\Sigma^*}$ is isomorphic to $\Sigma^*$ and not to $\kar{\Sigma}^*$. This implies a homotopy basis of $\kar{\Sigma^*}$ is not a homotopy basis of $\kar{\Sigma}^*$ in general.

\subsubsection{Example}\label{part2} Let $M$ be the monoid and  $\Sigma$ the 2-polygraph defined in example \ref{part1}. By Squier's Theorem \cite[Theorem 5.2]{Squier94}, a homotopy basis of the free $(2,1)$-category $\Sigma^\top$ is given by a 3-cell from $ab\alpha$ to $\alpha ba$. Thus, by \ref{result}, a homotopy basis of $\kar{\Sigma}^\ell$ is given by the following set of 3-cells:
$$\kar{\Sigma}_3=\{ ab\alpha \Rrightarrow \alpha ba,ab\alpha \Rrightarrow \alpha \pi_Y, \iota_X\alpha \Rrightarrow \alpha ba,\iota_X\alpha \Rrightarrow \alpha \pi_Y \}.$$
The $2$-polygraph $\kar{\Sigma}$ is Tietze equivalent to the convergent $2$-polygraph $\mathrm{Conv}$ defined by:
\begin{itemize}
\item $\mathrm{Conv}_0=\{ O,X,Y\}$,
\item $\mathrm{Conv}_1=\{ O \ofl{a} O,O \ofl{b} O,O \ofl{p_X} X,X \ofl{i_X} O,O \ofl{p_Y} Y,Y \ofl{i_Y} O\}$,
\item $\mathrm{Conv}_2=\{ aba \odfl{\alpha} a,p_Xi_X \odfl{\pi_X} ab,i_Xp_X \odfl{\iota_X} 1_X,p_Yi_Y \odfl{\pi_Y} ba,i_Yp_Y \odfl{\iota_Y} 1_Y, abp_X \odfl{} p_X, i_Xab \odfl{} i_X, bap_Y \odfl{} p_Y, i_Yba \odfl{} i_Y \}$.
\end{itemize}


\subsection{The linear case}

Let us recall the notions of $(n,p)$-linear categories introduced in \cite{All}.

\subsubsection{Linear $(n,p)$-categories}

A \emph{linear~$(n,0)$-category} is an internal $n$-category in the category $\mathbf{Mod}$ of modules over a given commutative ring. Let us assume linear~$(n,p)$-categories are defined for~$p \geqslant 0$. A linear~$(n+1,p+1)$-category is a data made of a set $\mathcal{C}_0$ and:
\begin{itemize}
\item for each $a$ and~$b$ in $\mathcal{C}_0$, a linear~$(n,p)$-category~$\mathcal{C}(a,b)$,
\item for each $a$ in $\mathcal{C}_0$, an identity morphism $i_a$ from the terminal $n$-category~$I_n$ to~$\mathcal{C}(a,a)$,
\item for each $a$,~$b$ and~$c$ in $\mathcal{C}_0$, a bilinear composition morphism $\star^{a,b,c}$ from $\mathcal{C}(a,b) \times \mathcal{C}(b,c)$ to~$\mathcal{C}(a,c)$.
\end{itemize}
such that:
\begin{itemize}
\item $\star^{a,c,d} \circ (\star^{a,b,c} \times id_{\mathcal{C}(c,d)})=\star^{a,b,d} \circ (id_{\mathcal{C}(a,b)} \times \star^{b,c,d})$,
\item $\star^{a,a,b} \circ (i_a \times id_{\mathcal{C}(a,b)}) \circ is_l=id_{\mathcal{C}(a,b)}=\star^{a,b,b} \circ (id_{\mathcal{C}(a,b)} \times i_q) \circ is_r$ where $is_l$ and~$is_r$ respectively denote the canonic isomorphisms from $\mathcal{C}(a,b)$ to~$I_n \times \mathcal{C}(a,b)$ and to~$\mathcal{C}(a,b) \times I_n$.
\end{itemize}

In particular, a linear $(n,n)$-category is a $n$-category $\mathcal{C}$ such that for each parallel $(n-1)$-cells $u$ and $v$ of $\mathcal{C}$, the set $\mathcal{C}_n(u,v)$ has a module structure over a ring making all compositions on $\mathcal{C}$ bilinear.

We call $\mathbf{LinCat}_{n,p}$ the category of linear $(n,p)$-categories. We also call $\mathbf{LinCat}_{n,p}^+$ the category of linear $(n,p)$-categories with a globular extension defined by the following pullback diagram:
\begin{gather*}
\begin{array}{c}
\tikz[scale=0.9]{
\node at (-0.5,0) {$\mathbf{LinCat}_{n,p}^+$};
\draw[color=black, ->] (0.5,0) -- (2,0);
\node at (2.9,0) {$\mathbf{Grph}_{n+1}$};
\draw[color=black, ->] (2.7,-0.5) -- (2.7,-2);
\node at (2.7,-2.5) {$\mathbf{Grph}_n$};
\draw[color=black, ->] (0,-0.5) -- (0,-2);
\node at (-0.5,-2.5) {$\mathbf{LinCat}_{n,p}$};
\draw[color=black, ->] (0.5,-2.5) -- (2,-2.5);
\node at (2.7,-2.5) {$\mathbf{Grph}_n$};
\draw[color=black] (0.3,-0.5) -- (0.55,-0.5) -- (0.55,-0.25);
\node at (1.25,-3) {$\mathcal{U}_{n,p}$};
\node at (3.2,-1.25) {$\mathcal{U}_n^G$};
} \end{array}
\end{gather*}
There is a forgetful functor from $\mathbf{LinCat}_{n,p}$ to the category $\mathbf{Cat}_n$ and this functor has a left adjoint. We can thus construct from an $n$-category a free $(n,p)$-linear category. The free linear $(n,p)$-category over an $n$-polygraph $\Sigma$ is the free linear $(n,p)$-category over the $n$-category $\Sigma^*$. We denote $\Sigma^\ell$ this linear $(n,p)$-category.

\subsubsection{Coherent presentation of a linear $(n,n)$-category}

We define the category $\mathbf{LinPol}_{n,p}$ of $(n,p)$-linear polygraphs and the functor~$\mathcal{F}_{n,p}$ from $\mathbf{LinPol}_{n,p}$ to~$\mathbf{LinCat}_{n,p}$ by induction on $n$ for $n \geqslant p$. $\mathbf{LinPol}_{n,n}$ is the category of $n$-polygraphs and the functor~$\mathcal{F}_{n,n}$ from $\mathbf{LinPol}_{n,n}$ to~$\mathbf{LinCat}_{n,n}$ is the free functor from $\mathbf{LinPol}_{n,n}$ to~$\mathbf{LinCat}_{n,n}$. Let us assume that the category~$\mathbf{LinPol}_{n,p}$ of linear~$(n,p)$-polygraphs and the functor~$\mathcal{F}_{n,p}$ from $\mathbf{LinPol}_{n,p}$ to~$\mathbf{LinCat}_{n,p}$ are defined. The category~$\mathbf{LinPol}_{n+1,p}$ is defined by the following pullback diagram:
\begin{gather*}
\begin{array}{c}
\tikz[scale=0.9]{
\node at (-0.5,0) {$\mathbf{LinPol}_{n+1,p}$};
\draw[color=black, ->] (0.7,0) -- (5,0);
\node at (5.9,0) {$\mathbf{Grph}_{n+1}$};
\draw[color=black, ->] (5.7,-0.5) -- (5.7,-2);
\node at (5.7,-2.5) {$\mathbf{Grph}_n$};
\draw[color=black, ->] (0,-0.5) -- (0,-2);
\node at (-0.5,-2.5) {$\mathbf{LinPol}_{n,p}$};
\draw[color=black, ->] (0.5,-2.5) -- (1.6,-2.5);
\node at (2.5,-2.5) {$\mathbf{LinCat}_{n,p}$};
\draw[color=black, ->] (3.4,-2.5) -- (5,-2.5);
\draw[color=black] (0.3,-0.5) -- (0.55,-0.5) -- (0.55,-0.25);
\node at (4,-3) {$\mathcal{U}_{n,p}$};
\node at (1,-3) {$\mathcal{F}_{n,p}$};
\node at (6.2,-1.25) {$\mathcal{U}_n^G$};
\node at (-0.5,-1.25) {$\mathcal{U}_{n,p}^P$};
\node at (2.85,0.5) {$\mathcal{U}_{n+1,p}^{GP}$};
} \end{array}
\end{gather*}
We denote by $\mathcal{F}_{n+1,p}^P$ the unique functor making the following diagram commutative:
\begin{gather*}
\begin{array}{c}
\tikz[scale=0.9]{
\node at (-0.6,0) {$\mathbf{LinCat}^+_{n,p}$};
\draw[color=black, ->] (0.5,0) -- (2,0);
\node at (2.9,0) {$\mathbf{Grph}_{n+1}$};
\draw[color=black, ->] (2.7,-0.5) -- (2.7,-2);
\node at (2.7,-2.5) {$\mathbf{Grph}_n$};
\draw[color=black, ->] (0,-0.5) -- (0,-2);
\node at (-0.6,-2.5) {$\mathbf{LinCat}_{n,p}$};
\draw[color=black, ->] (0.5,-2.5) -- (2,-2.5);
\node at (2.7,-2.5) {$\mathbf{Grph}_n$};
\draw[color=black] (0.3,-0.5) -- (0.55,-0.5) -- (0.55,-0.25);
\node at (1.25,-3) {$\mathcal{U}_{n,p}$};
\node at (3.2,-1.25) {$\mathcal{U}_n^G$};
\node at (-4,2) {$\mathbf{LinPol}_{n+1,p}$};
\draw[color=black, ->] (-2.6,1.5) -- (-0.6,0.4);
\draw[color=black, ->] (-4,1.5) -- (-4,-2);
\node at (-4,-2.5) {$\mathbf{LinPol}_{n,p}$};
\draw[color=black, ->] (-2.8,2) -- (2,0.4);
\node at (-2,0.6) {$\mathcal{F}_{n+1,p}^P$};
\draw[color=black, ->] (-3.1,-2.5) -- (-1.5,-2.5);
\node at (-2.4,-3) {$\mathcal{F}_{n,p}$};
\node at (-4.5,-0.25) {$\mathcal{U}_{n,p}^P$};
\node at (0.5,1.45) {$\mathcal{U}_{n+1,p}^{GP}$};
} \end{array}
\end{gather*}
The functor~$\mathcal{F}_{n+1,p}$ is defined as the following composite:
\begin{gather*}
\begin{array}{c}
\tikz[scale=0.9]{
\node at (-1,0) {$\mathbf{LinPol}_{n+1,p}$};
\draw[color=black, ->] (0.2,0) -- (2.2,0);
\node at (3.2,0) {$\mathbf{LinCat}_{n,p}^+$};
\draw[color=black, ->] (4,0) -- (6.2,0);
\node at (7.4,0) {$\mathbf{LinCat}_{n+1,p}$};
\node at (1.3,0.5) {$\mathcal{F}_{n+1,p}^P$};
\node at (5.1,0.5) {$\mathcal{F}_{n+1,p}^W$};
} \end{array}
\end{gather*}

A globular extension $\Gamma$ of the linear $(n,n)$-category~$\mathcal{C}$ is called a homotopy basis of~$\mathcal{C}$ if for any $n$-sphere $(f,g)$ of $\mathcal{C}$, the free linear $(n+1,n)$-category over~$(\mathcal{C},\Gamma)$ has an $(n+1)$-cell from $f$ to $g$. A coherent presentation of a linear $(n,n)$-category $\mathcal{C}$ is a linear $(n+2,n)$-polygraph $\Sigma$ such that the linear $(n+1,n)$-polygraph $\Sigma_{n+1}$ is a presentation of $\mathcal{C}$ and the set $\Sigma_{n+2}$ is a homotopy basis of $\Sigma_{n+1}^{\ell}$.

\subsubsection{Karoubi envelope of a linear $(n,n)$-category}

Let $\mathcal{C}$ be a linear $(n,n)$-category. In particular, $\mathcal{C}$ is an $n$-category. Let us denote by $\kar{\mathcal{C}}$ its Karoubi envelope. There is a structure of linear $(n,n)$-category on $\kar{\mathcal{C}}$ defined by $\alpha (e,\lambda f+g,e')=\lambda \alpha (e,f,e')+\alpha (e,g,e')$ for each scalar $\lambda$, each parallel $n$-cells $f$ and $g$ of $\mathcal{C}$ and each idempotents $e$ and $e'$ of $\mathcal{C}$ such that $f=e \star_{n-1} f \star_{n-1} e'$ and $g=e \star_{n-1} g \star_{n-1} e'$.

\subsubsection{Coherent presentation of the Karoubi envelope of a linear $(n,n)$-category}\label{reslin} Let $\mathcal{C}$ be a linear $(n,n)$-category and let $(\Sigma,\Sigma_{n+2})$ be a coherent presentation of $\mathcal{C}$. Let $\kar{\Sigma_{n+2}}$ be the globular extension of $\kar{\Sigma}$ defined as in \ref{result}. The $(n+2,n)$-polygraph $(\kar{\Sigma},\kar{\Sigma_{n+2}})$ is a coherent presentation of the Karoubi envelope of $\mathcal{C}$.

\begin{proof} To prove this proposition, we just have to prove that for each paralell $(n+1)$-cells $f$ and $g$ and each paralell $(n+1)$-cells $f'$ and $g'$ of $\kar{\Sigma}^\ell$ and each scalar $\lambda$ such that $\lambda f+f'$ and $\lambda g+g'$ are defined, we can construct an $(n+2)$-cell from $\lambda f+f'$ to $\lambda g+g'$ in $\kar{\Sigma_{n+2}}^\ell$ if there is an $(n+2)$-cell $A$ from $\mathrm{CS}(f)$ to $\mathrm{CS}(g)$ and an $(n+2)$-cell $A'$ from $\mathrm{CS}(f')$ to $\mathrm{CS}(g')$ in $\Sigma_{n+2}$. There is an $(n+2)$-cells $B$ from $f$ to $g$ in $\mathrm{CS}^{-1}(A)$ and is an $(n+2)$-cells $B'$ from $f'$ to $g'$ in $\mathrm{CS}^{-1}(A')$. Then, $\kar{\Sigma_{n+2}}^\ell$ contains the $(n+2)$-cell $\lambda B+B'$ from $\lambda f+f'$ to $\lambda g+g'$. This concludes the proof. \end{proof}

\subsubsection{Example}\label{algebra} The group $\mathbb{C}$-algebra $A$ of the symmetric group $\mathfrak{S}_2$ is a linear $(1,1)$-category presented by a linear $(2,1)$-polygraph $\Sigma$ defined by:
\begin{itemize}
\item $\Sigma_0$ has only one 0-cell,
\item $\Sigma_1$ has a 1-cell $s$,
\item $\Sigma_2$ has a 2-cell $\alpha$ from $ss$ to 1.
\end{itemize}
The linear $(2,2)$-category $A$ has three minimal idempotents: 0, $\frac{1-s}{2}$ and $\frac{1+s}{2}$. Thus, by \ref{pol}, the Karoubi envelope of $A$ is presented by the linear $(2,2)$-polygraph $\kar{\Sigma}$ defined by:
\begin{itemize}
\item $\kar{\Sigma}_0=\{ O,N,X,Y\}$,
\item $\kar{\Sigma}_1=\{ O \ofl{s} O,O \ofl{p_N} N,N \ofl{i_N} O,O \ofl{p_X} X,X \ofl{i_X} O,O \ofl{p_Y} Y,Y \ofl{i_Y} O\}$,
\item $\kar{\Sigma}_2=\{ ss \odfl{\alpha} s,p_Ni_N \odfl{\pi_N} 0,i_Np_N \ofl{\iota_N} 1_N,p_Xi_X \odfl{\pi_X} \frac{1-s}{2},i_Xp_X \odfl{\iota_X} 1_X,p_Yi_Y \odfl{\pi_Y} \frac{1+s}{2},i_Yp_Y \odfl{\iota_Y} 1_Y\}$.
\end{itemize}
A homotopy basis of the free linear $(2,2)$-category $\Sigma^\ell$ is given by a 3-cell from $s\alpha$ to $\alpha s$. Thus, by \ref{result}, a homotopy basis of $\kar{\Sigma}^\ell$ is given by the following set of 3-cells:
$$\kar{\Sigma}_3=\{ s\alpha \Rrightarrow \alpha s,(1-2\pi_X)\alpha \Rrightarrow \alpha s,(2\pi_Y)\alpha \Rrightarrow \alpha s,s\alpha \Rrightarrow \alpha (1-2\pi_X),(1-2\pi_X)\alpha \Rrightarrow \alpha (1-2\pi_X),$$
$$(2\pi_Y)\alpha \Rrightarrow \alpha (1-2\pi_X),s\alpha \Rrightarrow \alpha (2\pi_Y),(1-2\pi_X)\alpha \Rrightarrow \alpha (2\pi_Y),(2\pi_Y)\alpha \Rrightarrow \alpha (2\pi_Y)\}.$$
The linear $(2,2)$-polygraph $\kar{\Sigma}$ is Tietze equivalent to the convergent linear $(2,2)$-polygraph $\mathrm{Conv}$ defined by:
\begin{itemize}
\item $\mathrm{Conv}_0=\{ O,N,X,Y\}$,
\item $\mathrm{Conv}_1=\{ O \ofl{p_X} X,X \ofl{i_X} O,O \ofl{p_Y} Y,Y \ofl{i_Y} O\}$,
\item $\mathrm{Conv}_2=\{ 1_N \dfl 0,i_Xp_X \odfl{\iota_X} 1_X,p_Yi_Y \odfl{\pi_Y} 1-p_Xi_X,i_Yp_Y \odfl{\iota_Y} 1_Y,i_Xp_Y\dfl 0,i_Yp_X\dfl 0\}$.
\end{itemize}

\section{Categorification of algebras}

In this section, we define the Grothendieck decategorification of an $(n,n)$-category and the Grothendieck decategorification of a linear $(n+1,n)$-polygraph. We prove that the Grothendieck decategorification of a linear $(n+1,n)$-polygraph $\Sigma$ presents the Grothendieck decategorification of the linear $(n,n)$-category presented by $\Sigma$. We finally prove that the semi-convergence of the Grothendieck decategorification of a linear $(n+1,n)$-polygraph $\Sigma$ is equivalent to the uniqueness of decompositions as a direct sum of indecomposable $(n-1)$-cells in the linear $(n,n)$-category presented by $\Sigma$ up to isomorphism.

\subsection{Grothendieck decategorification}

The Grothendieck decategorification of a linear category $\mathcal{C}$ is the group generated by the isomorphism classes of $\mathcal{C}$ and subject to the relation $[a]=[b]+[c]$ whenever an object $a$ is direct sum of two objects $b$ and $c$. If the category $\mathcal{C}$ is monoidal, this case corresponding to a linear $(2,2)$-category the Grothendieck decategorification of $\mathcal{C}$ is also a ring with a product defined by $[a\otimes b]=[a][b]$ for any objects $a$ and $b$ of $\mathcal{C}$. In this section, we extend the definition of Grothendieck decategorification to arbitrary linear $(n,n)$-categories and give a construction for presenting such Grothendieck decategorifications.

\subsubsection{Direct sums in a linear $(n,n)$-category} Let $n>1$ be an integer and $\mathcal{C}$ be a linear $(n,n)$-category. Let $a$, $b$ and $c$ be $(n-1)$-cells of $\mathcal{C}$. We say that $a$ is \emph{direct sum} of $b$ and $c$ if there exist $n$-cells $a \ofl{p_b} b$, $a \ofl{p_c} c$, $b \ofl{i_b} a$ and $c \ofl{i_c} a$ such that:
\begin{itemize}
\item $p_b \star_{n-1} i_b+p_c \star_{n-1} i_c=1_a$,
\item $i_b \star_{n-1} p_b=1_b$,
\item $i_c \star_{n-1} p_c=1_c$.
\end{itemize}
In this case, we denote $a \simeq b \oplus c$.

\subsubsection{Grothendieck decategorification of a linear $(n,n)$-category} Let $n>1$ be an integer and $\mathcal{C}$ be a linear $(n,n)$-category. Two $(n-1)$-cells $u$ and $v$ of $\mathcal{C}$ are \emph{isomorphic} if there is an $n$-cell from $u$ to $v$ which is invertible for the $n$-composition. We will call $[u]$ the isomorphism class of the $(n-1)$-cell $u$. The \emph{Grothendieck decategorification} of $\mathcal{C}$ is the linear $\mathbb{Z}$-linear $(n-1,n-1)$-category $K(\mathcal{C})$ defined by:
\begin{itemize}
\item for $k<n-1$, the linear $\mathbb{Z}$-linear $(n-1,n-1)$-category $K(\mathcal{C})$ has the same $k$-cells than $\mathcal{C}$,
\item for any parallel $(n-2)$-cells $x$ and $y$ of $K(\mathcal{C})$, the $\mathbb{Z}$-module $K(\mathcal{C})_{n-1}[x,y]$ is the free abelian group generated by the isomorphisms classes of $\mathcal{C}_{n-1}[x,y]$ and subject to the relation $[a]=[b]+[c]$ for each $(n-1)$-cells $a$, $b$ and $c$ such that $a \simeq b \oplus c$
\item for any $0 \leqslant k \leqslant n-2$ and any $k$-composable $(n-1)$-cells $u$ and $v$ of $\mathcal{C}_{n-1}$, we have $[u] \star_k [v]=[u \star_k v]$.
\end{itemize}

\subsubsection{Example} Let $\mathcal{M}$ be a linear $(2,2)$-category with only one 0-cell. The Grothendieck decategorification $K(\mathcal{M})$ of $\mathcal{M}$ is an abelian group with a $\mathbb{Z}$-bilinear associative composition map $\star_0$. Thus, $K(\mathcal{M})$ is a ring.

\subsubsection{Remark} In general, given an $(n,n)$-category $\mathcal{C}$, the Grothendieck decategorifications $K(\mathcal{C})$ and $K(\kar{\mathcal{C}})$ are not isomorphic. For example, the $\mathbb{C}$-algebra $A$ from Example \ref{algebra} has a Grothendieck decategorification isomorphic to $\mathbb{Z}$ whereas the Grothendieck decategorification of $\kar{A}$ is isomorphic to $\mathbb{Z}^2$. If all idempotents of the $(n,n)$-category $\mathcal{C}$ are split, we have an isomorphism between $K(\mathcal{C})$ and $K(\kar{\mathcal{C}})$.

\subsubsection{Isomorphism proofs} Let $\Sigma$ be a linear $(n+1,n)$-polygraph. Let $u$ and $v$ be distinct $(n-1)$-cells of the free linear $(n+1,n)$-category $\Sigma^\ell$. An \emph{isomorphism proof} between $u$ and $v$ is a data $(\alpha_u,\alpha_v)$ made of two $(n+1)$-cells in $\Sigma^\ell$ such that there exist $n$-cells $u \ofl{a_u} v$ and $v \ofl{a_v} u$ verifying:
\begin{itemize}
\item $\alpha_u$ is an $(n+1)$-cell from $a_u \star_{n-1} a_v$ to $1_u$,
\item $\alpha_v$ is an $(n+1)$-cell from $a_v \star_{n-1} a_u$ to $1_v$.
\end{itemize}
An isomorphism proof $(\alpha_u,\alpha_v)$ is \emph{minimal} if there is no $(n-1)$-cell $w$, no integer $k<n-1$ and no isomorphism proof $(\alpha_{u'},\alpha_{v'})$ other than $(\alpha_u,\alpha_v)$  such that $(\alpha_u,\alpha_v)=(w \star_{n-1}\alpha_{u'},w \star_{n-1}\alpha_{v'})$ or $(\alpha_u,\alpha_v)=(\alpha_{u'}\star_{n-1} w,\alpha_{v'}\star_{n-1} w)$.

\subsubsection{Direct sum proofs} Let $\Sigma$ be a linear $(n+1,n)$-polygraph. Let $a$, $b$ and $c$ be $(n-1)$-cells of the free linear $(n+1,n)$-category $\Sigma^\ell$. A \emph{direct sum proof} of $a \simeq b \oplus c$ is a data $(\alpha_a,\alpha_b,\alpha_c)$ made of three $(n+1)$-cells in $\Sigma^\ell$ such that there exist $n$-cells $a \ofl{p_b} b$, $a \ofl{p_c} c$, $b \ofl{i_b} a$ and $c \ofl{i_c} a$ in $\Sigma^\ell$ verifying:
\begin{itemize}
\item $\alpha_a$ is an $(n+1)$-cell from $p_b \star_{n-1} i_b+p_c \star_{n-1} i_c$ to $1_a$,
\item $\alpha_b$ is an $(n+1)$-cell from $i_b \star_{n-1} p_b$ to $1_b$,
\item $\alpha_c$ is an $(n+1)$-cell from $i_c \star_{n-1} p_c$ to $1_c$.
\end{itemize}
A direct sum proof $(\alpha_a,\alpha_b,\alpha_c)$ is said to be \emph{minimal} if there are no $(n-1)$-cell $u$, no integer $k<n-1$ and no direct sum proof $(\alpha_{a'},\alpha_{b'},\alpha_{c'})$ other than $(\alpha_a,\alpha_b,\alpha_c)$ such that $(\alpha_a,\alpha_b,\alpha_c)=(u \star_{n-1}\alpha_{a'},u\star_{n-1} \alpha_{b'},u\star_{n-1} \alpha_{c'})$ or $(\alpha_a,\alpha_b,\alpha_c)=(\alpha_{a'}\star_{n-1} u,\alpha_{b'}\star_{n-1} u,\alpha_{c'}\star_{n-1} u)$.

\subsubsection{Grothendieck decategorification of a linear $(n+1,n)$-polygraph} Let $\mathcal{C}$ be a linear $(n,n)$-category presented by a linear $(n+1,n)$-polygraph $\Sigma$. The \emph{Grothendieck decategorification} of $\Sigma$ is the linear $\mathbb{Z}$-linear $(n,n-1)$-polygraph $K(\Sigma)$ defined by:
\begin{itemize}
\item for $k \leqslant n-1$, the linear $\mathbb{Z}$-linear $(n,n-1)$-polygraph $K(\Sigma)$ has the same $k$-cells than $\Sigma$,
\item for each $(n-1)$-cells $u$ and $v$ of $\Sigma_{n-1}^\ell$ such that $u \neq v$ and there is a minimal isomorphism proof between $u$ and $v$, there is an $n$-cell in $K(\Sigma)$ from $u$ to $v$.
\item for each $(n-1)$-cells $a$, $b$ and $c$ of $\Sigma_{n-1}^\ell$ such that there is a minimal direct sum proof of $a \simeq b \oplus c$, there is an $n$-cell in $K(\Sigma)$ from $a$ to $b+c$.
\end{itemize}

\subsubsection{Theorem}\label{presentation} Let $\mathcal{C}$ be a linear $(n,n)$-category presented by a linear $(n+1,n)$-polygraph $\Sigma$. The Grothendieck decategorification $K(\Sigma)$ of $\Sigma$ presents the Grothendieck decategorification $K(\mathcal{C})$ of $\mathcal{C}$.

\begin{proof} By definition, the $(n-1,n-1)$-category presented by $K(\Sigma)$ has the same $k$-cells than $K(\mathcal{C})$ for $k<n-1$. This linear $(n-1,n-1)$-category is also generated by the same $(n-1)$-cells than $K(\mathcal{C})$. Each relation verified by the $(n-1)$-cells of the $(n-1,n-1)$-category presented by $K(\Sigma)$ is also verified by the $(n-1)$-cells of $K(\mathcal{C})$. Let us now prove that each relation verified by the $(n-1)$-cells of $K(\mathcal{C})$ is verified by the $(n-1)$-cells of the $(n-1,n-1)$-category presented by $K(\Sigma)$.

Let $a\simeq b \oplus c$ be a direct sum in $\mathcal{C}$. If there is a minimal proof of this direct sum, then there is an $n$-cell in $K(\Sigma)_n^\ell$ from $[a]$ to $[b]+[c]$. Else, there are decompositions:
\begin{itemize}
\item $a=u_1 \star_{n-2} (u_2 \star_{n-3}( \cdots (u_{n-1} \star_0 a' \star_0 u_n) \cdots )\star_{n-3} u_{2n-3}) \star_{n-2} u_{2n-2}$,
\item $b=u_1 \star_{n-2} (u_2 \star_{n-3}( \cdots (u_{n-1} \star_0 b' \star_0 u_n) \cdots )\star_{n-3} u_{2n-3}) \star_{n-2} u_{2n-2}$,
\item $c=u_1 \star_{n-2} (u_2 \star_{n-3}( \cdots (u_{n-1} \star_0 c' \star_0 u_n) \cdots )\star_{n-3} u_{2n-3}) \star_{n-2} u_{2n-2}$
\end{itemize}
where all $u_i$ are $(n-1)$-cells of $\mathcal{C}$ and the direct sum $a'\simeq b' \oplus c'$ has a minimal proof. Hence, by distributivity of the compositions there is an $n$-cell in $K(\Sigma)_n^\ell$ from $[a]$ to $[b]+[c]$. This concludes the proof. \end{proof}

\subsubsection{Example}\label{notKS} Let $\Sigma$ be the linear $(2,1)$-polygraph defined by:
\begin{itemize}
\item $\Sigma_0$ has five 0-cells $O$, $X_1$, $X_2$, $Y_1$ and $Y_2$, 
\item $\Sigma_1$ has the 1-cells $O \ofl{a} O$, $O \ofl{b} O$, $O \ofl{p_{X_1}} X_1$, $O \ofl{p_{X_2}} X_2$, $O \ofl{p_{Y_1}} X_1$, $O \ofl{p_{Y_2}} X_2$, $X_1 \ofl{i_{X_1}} O$, $X_2 \ofl{i_{X_2}} O$, $Y_1 \ofl{i_{Y_1}} O$ and $Y_2 \ofl{i_{Y_2}} O$,
\item $\Sigma_2$ has the 2-cells $aba \odfl{\alpha} a$, $p_{X_1}i_{X_1} \odfl{\pi_{X_1}} ab$, $p_{X_2}i_{X_2} \odfl{\pi_{X_2}} 1_O-ab$, $i_{X_1}p_{X_1} \odfl{\iota_{X_1}} 1_{X_1}$, $i_{X_2}p_{X_2} \odfl{\iota_{X_2}} 1_{X_2}$, $p_{Y_1}i_{Y_1} \odfl{\pi_{Y_1}} ba$, $p_{Y_2}i_{Y_2} \odfl{\pi_{Y_2}} 1_Oba$, $i_{Y_1}p_{Y_1} \odfl{\iota_{Y_1}} 1_{Y_1}$ and $i_{Y_2}p_{Y_2} \odfl{\iota_{Y_2}} 1_{Y_2}$.
\end{itemize}
Let $\mathcal{C}$ be the linear $(1,1)$-category presented by $\Sigma$. The Grothendieck decategorification $K(\Sigma)$ of $(2,1)$-polygraph $\Sigma$ has five 0-cells $[O]$, $[X_1]$, $[X_2]$, $[Y_1]$ and $[Y_2]$. Two direct sums in $\mathcal{C}$ have a minimal proof: $O \simeq X_1 \oplus X_2$ with the proof $(\pi_{X_1}+\pi_{X_2},\iota_{X_1},\iota_{X_2})$ and $O \simeq Y_1 \oplus Y_2$ with the proof $(\pi_{Y_1}+\pi_{Y_2},\iota_{Y_1},\iota_{Y_2})$. Then $K(\Sigma)$ has the 1-cells $[O] \ofl{}[X_1] + [X_2]$ and $[O] \ofl{}[Y_1] + [Y_2]$. Because $\Sigma$ does not have any isomorphism proof, the Grothendieck decategorification $K(\mathcal{C})$ of $\mathcal{C}$ is the free abelian group over three elements.

\subsection{Krull-Schmidt linear $(n,n)$-categories}

A category $\mathcal{C}$ with direct sums is said to be Krull-Schmidt if any object of $\mathcal{C}$ can be uniquely decomposed as a direct sum of indecomposable objects and those indecomposable objects have local endomorphisms rings. We extend this notion to linear $(n,n)$-categories and give a criterion to decide if a linear $(n,n)$-category verifies the first part of the Krull-Schmidt property given a presentation of this linear $(n,n)$-category.

\subsubsection{Krull-Schmidt linear $(n,n)$-categories} Let $\mathcal{C}$ be a linear $(n,n)$-category. An \emph{indecomposable} of $\mathcal{C}$ is an $(n-1)$-cells of  $\mathcal{C}$ without non trivial decomposition into a direct sum. We say that $\mathcal{C}$ is \emph{Krull-Schmidt} if all $(n-1)$-cell of $\mathcal{C}$ can be decomposed into a unique direct sums of indecomposable $(n-1)$-cells of $\mathcal{C}$ up to isomorphism and the ring $(\mathcal{C}_n(e),+,\star_{n-1})$ is local for each $(n-1)$-cell $e$ in this decomposition.

\subsubsection{Example} Let $\mathcal{C}$ be the linear $(1,1)$-category presented by the linear $(2,1)$-polygraph $\Sigma$ of Example \ref{notKS}. This $(1,1)$-category is not Krull-Schmidt.

\subsubsection{Rewriting steps} Let $\Sigma$ be a linear linear $(n,n-1)$-polygraph. A \emph{rewriting step} of $\Sigma$ is an $n$-cell of $\Sigma^\ell$ of the form:
$$1_{u_1} \star_{n-1} \cdots  (1_{u_{n-1}} \star_0 \lambda\alpha \star_0 1_{u_{n}}) \cdots  \star_n 1_{u_{2n}}+v$$
where $\alpha$ is in $\Sigma_n$, the $u_i$ are monomials of $\Sigma$, $\lambda$ is a nonzero scalar and $v$ is an $(n-1)$-cell of $\Sigma^\ell$ such that $1_{u_1} \star_{n-1} \cdots  (1_{u_{n-1}} \star_0 s_{n-1}(\alpha) \star_0 1_{u_{n}}) \cdots  \star_n 1_{u_{2n}}$ does not appear in the monomial decomposition of $v$. A \emph{rewriting sequence} of $\Sigma$ is a finite or infinite sequence $f_0\cdot f_1 \dot \ldots \dot f_i \dot\, \cdots$, where the $f_i$ are rewriting steps such that $t_1(f_i)=s_1(f_{i+1})$ for all $i\geq 0$. An $(n-1)$-cell $u$ of $\Sigma^\ell$ \emph{rewrites} into an $(n-1)$-cell $v$ of $\Sigma^\ell$ if there is a rewriting sequence from $u$ to $v$.

\subsubsection{Confluence} A \emph{branching} of $\Sigma$ is a pair of rewriting sequences of $\Sigma$ with the same source. A finite branching $(\alpha ,\beta)$ is \emph{confluent} if it there exist two rewriting sequences of $\Sigma$ respectively of source $t_{n-1}(\alpha)$ and $t_{n-1}(\beta)$ with the same target. We say that $\Sigma$ is confluent if all finite branchings of $\Sigma$ are confluent.

\subsubsection{Normal forms} Let $u$ be an $(n-1)$-cell of $\Sigma^\ell$. A \emph{normal form} of $u$ is an $(n-1)$-cell $v$ of $\Sigma^\ell$ such that $u$ rewrites into $v$ and $v$ cannot be rewritten. A \emph{quasi-normal form} of $u$ is an $(n-1)$-cell $v$ of $\Sigma^\ell$ such that $u$ rewrites into $v$ and for each $v'$ such that $v$ rewrites into $v'$, we have $v'$ rewrites into $v$.

\subsubsection{Termination and convergence} We say that $\Sigma$ is \emph{terminating} if it has no infinite rewriting sequence, that is there is no sequence $(u_k)_{k \in \mathbb{N}}$ of $(n-1)$-cells such that for each $i$ in $\mathbb{N}$, there is a rewriting step from $u_i$ to $u_{i+1}$. We say that $\Sigma$ is \emph{quasi-terminating} \cite{Der} if each sequence $(u_k)_{k \in \mathbb{N}}$ of $(n-1)$-cells such that for each $i$ in $\mathbb{N}$, there is a rewriting step from $u_i$ to $u_{i+1}$, contains an infinite occurrence of the same $(n-1)$-cell. We say that $\Sigma$ is \emph{(quasi-)convergent} if it is (quasi-)terminating and confluent.

\subsubsection{Remark} If $\Sigma$ is convergent, each $(n-1)$-cell of $\Sigma^\ell$ has a unique normal form. If $\Sigma$ is quasi-convergent, each $(n-1)$-cell of $\Sigma^\ell$ has at least one quasi-normal form and all those quasi-normal forms rewrite into each other.

\subsubsection{Theorem}\label{thKS} Let $\mathcal{C}$ be a linear $(n,n)$-category presented by a linear $(n+1,n)$-polygraph $\Sigma$. Then, each $(n-1)$-cell of $\mathcal{C}$ has a unique decomposition into a direct sum of indecomposable $(n-1)$-cells up to isomorphism if and only if the Grothendieck decategorification $K(\Sigma)$ of $\Sigma$ is quasi-convergent.

\begin{proof} The first implication is obvious. We assume now $K(\Sigma)$ is quasi-convergent. Let $f$ be a rewriting step of $K(\Sigma)$ such that $t_{n-1}(f)$ does not rewrite into $s_{n-1}(f)$. Then $t_{n-1}(f)$ is the sum of two $(n-1)$-cells of $\Sigma^\ell$. Thus, any $(n-1)$-cell of $K(\Sigma)_{n-1}^*=\Sigma_{n-1}^\ell$ has a quasi-normal form of the form $u_1+ \cdots u_k$ where each $u_i$ is a monomial of $K(\Sigma)$ which cannot be rewritten into a sum of two monomials. This concludes the proof. \end{proof}

\begin{small}
\renewcommand{\refname}{\Large\textsc{References}}
\bibliographystyle{alpha}
\bibliography{Biblio}
\end{small}

\end{document}